\documentclass[12pt,twoside]{article}
\usepackage{amsfonts}
\usepackage{amsmath}
{\markboth{\sl  C. Abdelkefi}{\sl
 Weighted function spaces and Dunkl transform}
\newtheorem{theorem}{Theorem}[section]
\newtheorem{remark}{Remark}[section]
\newtheorem{lemma}{Lemma}[section]
\newtheorem{proposition}{Proposition}[section]
\newtheorem{corollary}{Corollary}[section]
\numberwithin {equation}{section}
\newenvironment{proof}{\textbf{Proof.}} {\hfill $\Box$}

\begin{document}
\title{\bf Weighted function spaces and Dunkl transform}
\author{Chokri Abdelkefi
 \footnote{\small Work supported
by the DGRST research project 04/UR/15-02 and the program CMCU 10G /
1503.}\\\\\small  Department of Mathematics, Preparatory
Institute of Engineer Studies of Tunis  \\ \small 1089 Monfleury Tunis, Tunisia\\
\small E-mail : chokri.abdelkefi@ipeit.rnu.tn  }%
\date{}
\maketitle
\begin{abstract}
We introduce first weighted function spaces on $\mathbb{R}^d$ using
the Dunkl convolution that we call Besov-Dunkl spaces. We provide
characterizations of these spaces by decomposition of functions.
Next we obtain in the real line and in radial case on $\mathbb{R}^d$
weighted $L^p$-estimates of the Dunkl transform of a function in
terms of an integral modulus of continuity which gives a
quantitative form of the Riemann-Lebesgue lemma. Finally, we show in
both cases that the Dunkl transform of a function is in $L^1$ when
this function belongs to a suitable Besov-Dunkl space.
\end{abstract}
\section{Introduction }
\par Dunkl theory generalizes classical Fourier analysis on
$\mathbb{R}^d$. It started twenty years ago with Dunkl's seminal
work [9] and was further developed by several mathematicians (see
[2, 8, 11, 12, 13, 21]). We consider Dunkl operators $T_i, 1 \leq i
\leq d$, on $\mathbb{R}^d$, associated to an arbitrary finite
reflection group $G$ and a non negative multiplicity function $k$.
The Dunkl kernel $E_k$ has been introduced by C.F. Dunkl in [10].
This kernel is used to define the Dunkl transform $\mathcal{F}_k$.
K. Trim\`eche has introduced in [22] the Dunkl translation operators
$\tau_x$,
 $x\in \mathbb{R}^d$, on the space of infinitely
differentiable functions on $\mathbb{R}^d$.  At the moment an
explicit formula for the Dunkl translation $\tau_x(f)$ of a function
$f$ is unknown in general. However, such formula is known in two
cases: when the function $f$ is radial and when $G= \mathbb{Z}_2^d$,
(see next section). In particular, the boundedness of $\tau_x$ is
established in these cases. As a result one obtains a formula for
the convolution
$\ast_k$.\\

There are many ways to define the Besov spaces (see [4, 5, 20]) and
the Besov spaces for the Dunkl operators (see [1, 3]). Let $
\Phi(\mathbb{R}^d)$ be the set of all sequences $
(\varphi_{j})_{j\in \mathbb{N}}$ of functions in $
\mathcal{S}(\mathbb{R}^d)^{rad}$ such that
\begin{itemize}
\item[(i)] $\mbox{supp}\,\mathcal{F}_k(\varphi_{0})
\subset B(0,2)=\{x\in \mathbb{R}^d\,;\, \|x\|\leq 2\}$ and\\
$\mbox{supp}\,\mathcal{F}_k(\varphi_{j}) \subset A_j=\{x\in
\mathbb{R}^d\,;\,2^{j-1}\leq \|x\|\leq 2^{j+1}\}$ for $j\in
\mathbb{N\backslash}\{0\}$.
\item[(ii)] $ \displaystyle{\sup_{j\in \mathbb{N}}\| \varphi_{j}\|_{1,k}<+\infty}$.
\item[(iii)] $\displaystyle{\sum_{j\in \mathbb{N}}\mathcal{F}_k(\varphi_{j})(x) =1}$,  for  $x\in \mathbb{R}^d $.
\end{itemize}$ \mathcal{S}(\mathbb{R}^d)^{rad}$ being the subspace of functions
in the Schwartz space $\mathcal{S}(\mathbb{R}^d)$ which are radial.
For $ \beta>0$, $1\leq p , q \leq + \infty$ and $
(\varphi_{j})_{j\in \mathbb{N}}\in \Phi(\mathbb{R}^d)$, we define
the weighted Besov-Dunkl space denoted by $
\mathcal{B}\mathcal{D}_{p,q}^{\beta,k} $ as the subspace of
functions $f \in L^p_k(\mathbb{R}^d)$ satisfying
 $$
\Big(\sum_{j\in\mathbb{N}}(2^{j\beta}\|\varphi_{j }\ast_k
f\|_{p,k})^q \Big)^{\frac{1}{q}}<+\infty \;\;\;\quad if\quad q <
+\infty$$ and $$ \sup_{j\in \mathbb{N}}2^{j\beta}\|\varphi_{j
}\ast_k f\|_{p,k}<+\infty\quad\quad if \quad q = +\infty ,$$ where
$L^p_k(\mathbb{R}^d)$ is the space $L^{p}(\mathbb{R}^d, w_k(x)dx),$
with $w_k$ the weight function given by
$$w_k(x) = \prod_{\xi\in R_+} |\langle \xi,x\rangle|^{2k(\xi)},
\quad x \in \mathbb{R}^d ,$$ $R_+$ being a positive root system (see next section).\\

For $1\leq p\leq 2$, we introduce the class
$\mathfrak{M}_p(\mathbb{R}^d)$ of sequences $ (g_{j})_{j\in
\mathbb{N}}$ of functions in $L^p_k(\mathbb{R}^d)$ such that
$\mbox{supp}\,\mathcal{F}_k(g_{0}) \subset B(0,2)$ and
$\mbox{supp}\,\mathcal{F}_k(g_{j}) \subset A_j$ for $j\in \mathbb{N\backslash}\{0\}$.\\

The modulus of continuity $\omega_{p}(f)$ of an
$L^p$-function $f$, $1\leq p\leq2$ is defined:\\
$\bullet$ On the real line by :$$\omega_{p}(f)(t)= \|\tau_t(f) +
\tau_{-t}(f) - 2f\|_{p,k}\,,\;\;t\geq0\,,$$
which is known as the modulus of continuity of second order of $f$. \\
$\bullet$ In radial case on $\mathbb{R}^d$ by : $$\omega_p(f) (t) =
\displaystyle{\int_{S^{d-1}}}\|\tau_{tu}(f) -
f\|_{p,k}d\sigma(u)\,,\;\;t\geq0\,,$$ where $f$ is a radial
$L^p$-function and $S^{d-1}$ the unit sphere on $\mathbb{R}^d$ with
the normalized surface measure $d\sigma$.\\

In this paper,  we provide first characterizations of the
Besov-Dunkl spaces $\mathcal{B}\mathcal{D}_{p,q}^{\beta,k}$  by
decomposition of functions using the class
 $\mathfrak{M}_p(\mathbb{R}^d)$ for $ \beta>0$, $1\leq p\leq 2$ and
$1\leq q < + \infty$. This extend to the Dunkl operators on
$\mathbb{R}^d$ some results obtained for the classical case in [20].
Next we obtain weighted $L^p$-estimates of the Dunkl transform of a
function in terms of an integral modulus of continuity. These
results are carried out on the real line and in radial case on
$\mathbb{R}^d$. In both cases these are expressed as a gauge on the
size of in terms of an integral modulus of continuity of $f$. As
consequence, we give a quantitative form of the Riemann-Lebesgue
lemma. Finally, we show that the Dunkl transform $\mathcal{F}_k(f)$
of a function $f$ is in $L^1_k(\mathbb{R}^d)$
when $f$ belongs to a suitable Besov-Dunkl space.  \\

The contents of this paper are as follows. \\In section 2, we
collect some basic definitions and results about harmonic analysis
associated with Dunkl operators .\\
In section 3, we characterize the Besov-Dunkl spaces by
decomposition of functions. \\In section 4, we obtain inequalities
for the Dunkl transform $\mathcal{F}_k(f)$ of $f$, both on the real
line and in radial case on $\mathbb{R}^d$. As consequence, we give a
quantitative form of the Riemann-Lebesgue lemma. We show Finally
further results of integrability of $\mathcal{F}_k(f)$ when $f$ satisfies a suitable condition.\\

Along this paper we denote by $\langle .,.\rangle$ the usual
Euclidean inner product in $\mathbb{R}^d$ as well as its extension
to $\mathbb{C}^d \times \mathbb{C}^d$, we write for $x \in
\mathbb{R}^d, \|x\| = \sqrt{\langle x,x\rangle}$ and we use $c$ to
denote a suitable positive constant which is not necessarily the
same in each occurrence. Furthermore, we denote by

$\bullet\quad \mathcal{E}(\mathbb{R}^d)$ the space of infinitely
differentiable functions on $\mathbb{R}^d$.

$\bullet\quad \mathcal{S}(\mathbb{R}^d)$ the Schwartz space of
functions in $\mathcal{E}( \mathbb{R}^d)$ which are rapidly
decreasing as well as their derivatives.

$\bullet\quad \mathcal{D}(\mathbb{R}^d)$ the subspace of
$\mathcal{E}(\mathbb{R}^d)$ of compactly supported functions.
\section{Preliminaries}
 $ $ In this section, we recall some notations and
results in Dunkl theory and we refer for more details to the
articles [8, 9, 18] or to the surveys [16].\\

Let $G$ be a finite reflection group on $\mathbb{R}^{d}$, associated
with a root system $R$. For $\alpha\in R$, we denote by $H_\alpha$
the hyperplane orthogonal to $\alpha$. For a given
$\beta\in\mathbb{R}^d\backslash\bigcup_{\alpha\in R} H_\alpha$, we
fix a positive subsystem $R_+=\{\alpha\in R: \langle
\alpha,\beta\rangle>0\}$. We denote by $k$ a nonnegative
multiplicity function defined on $R$ with the property that $k$ is
$G$-invariant. We associate with $k$ the index
$$\gamma = \gamma (R) = \sum_{\xi \in R_+} k(\xi) \geq 0, $$
and the weight function $w_k$ defined by
$$w_k(x) = \prod_{\xi\in R_+} |\langle \xi,x\rangle|^{2k(\xi)},
\quad x \in \mathbb{R}^d .$$ $w_k$ is $G$-invariant and homogeneous
of degree $2\gamma$.\\

Further, we introduce the Mehta-type constant $c_k$ by
$$c_k = \left(\int_{\mathbb{R}^d} e^{- \frac{\|x\|^2}{2}}
w_k (x)dx\right)^{-1}.$$

For every $1 \leq p \leq + \infty$, we denote by
$L^p_k(\mathbb{R}^d)$ the space $L^{p}(\mathbb{R}^d, w_k(x)dx),$
$\,$ $L^p_k( \mathbb{R}^d)^{rad}$ the subspace of those $f \in
L^p_k( \mathbb{R}^d)$ that are radial  and  we use $\|\ \;\|_{p,k}$
as a shorthand for $\|\ \;\|_{L^p_k( \mathbb{R}^d)}.$\\

By using the homogeneity of $w_k$, it is shown in [14] that for $f
\in L^1_k ( \mathbb{R}^d)^{rad}$, there exists a function $F$ on
$[0, + \infty)$ such that $f(x) = F(\|x\|)$, for all $x \in
\mathbb{R}^d$. The function $F$ is integrable with respect to the
measure $r^{2\gamma+d-1}dr$ on $[0, + \infty)$ and we have
 \begin{eqnarray} \int_{\mathbb{R}^d}  f(x)\,w_k(x)dx &=&
 \int^{+\infty}_0
\Big( \int_{S^{d-1}}w_k(ry)d\sigma(y)\Big)
F(r)r^{d-1}dr\nonumber\\&= & d_k\int^{+ \infty}_0 F(r)
r^{2\gamma+d-1}dr,
\end{eqnarray}
  where $S^{d-1}$
is the unit sphere on $\mathbb{R}^d$ with the normalized surface
measure $d\sigma$  and \begin{eqnarray}d_k=\int_{S^{d-1}}w_k
(x)d\sigma(x) = \frac{c^{-1}_k}{2^{\gamma +\frac{d}{2} -1}
\Gamma(\gamma + \frac{d}{2})}\;.  \end{eqnarray}

The Dunkl operators $T_j\,,\ \ 1\leq j\leq d\,$, on $\mathbb{R}^d$
associated with the reflection group $G$ and the multiplicity
function $k$ are the first-order differential- difference operators
given by
$$T_jf(x)=\frac{\partial f}{\partial x_j}(x)+\sum_{\alpha\in\mathcal{R}_+}k(\alpha)
\alpha_j\,\frac{f(x)-f(\sigma_\alpha(x))}{\langle\alpha,x\rangle}\,,\quad
f\in\mathcal{E}(\mathbb{R}^d)\,,\quad x\in\mathbb{R}^d\,,$$ where
$\sigma_\alpha$ is the reflection on the hyperplane $H_\alpha$ and
$\alpha_j=\langle\alpha,e_j\rangle,$ $(e_1,\ldots,e_d)$ being
the canonical basis of $\mathbb{R}^d$\,.\\In the case $k=0$, the $T_j$ reduce to the corresponding partial derivatives.\\

For $y \in \mathbb{R}^d$, the system
$$\left\{\begin{array}{lll}T_ju(x,y)&=&y_j\,u(x,y),\qquad1\leq j\leq d\,,\\  &&\\
u(0,y)&=&1\,.\end{array}\right.$$ admits a unique analytic solution
on $\mathbb{R}^d$, denoted by $E_k(x,y)$ and called the Dunkl
kernel. This kernel has a unique holomorphic extension to
$\mathbb{C}^d \times \mathbb{C}^d $.\\ M. R\"osler has proved in
[15] the following integral representation for the Dunkl kernel
$$E_k(x,z)=\int_{\mathbb{R}^d}e^{\langle
y\,,\,z\rangle}\,d\mu_x^k(y)\,,\quad x\in\mathbb{R}^d,\quad
z\in\mathbb{C}^d\,,$$ where $\mu_x^k$ is a probability measure on
$\mathbb{R}^d$ with support in the closed ball $B(0,\|x\|)$ of
center $0$ and radius $\|x\|$. We have for $\lambda\in \mathbb{C}$
and $z, z'\in \mathbb{C}^d,\,
 E_k(z,z')=E_k(z',z)$, $E_k(\lambda z,z')=E_k(z,\lambda z')$ and
 $|E_k(x,iy)| \leq 1$ for $x, y
\in \mathbb{R}^d$. It was shown in [14, 17] that
\begin{eqnarray}\int_{S^{d-1}}E_k(ix,z)\,w_k
(z)d\sigma(z)=d_k\;j_{\gamma +\frac{d}{2}-1}(\|x\|),\quad x\in
 \mathbb{R}^d,
 \end{eqnarray} where $j_{\gamma +\frac{d}{2}-1}$ is the normalized Bessel function of the first kind
and order $\gamma + \frac{d}{2}-1$.\\

The Dunkl transform $\mathcal{F}_k$ is defined for $f \in
\mathcal{D}( \mathbb{R}^d)$ by
$$\mathcal{F}_k(f)(x) =c_k\int_{\mathbb{R}^d}f(y) E_k(-ix, y)w_k(y)dy,\quad
x \in \mathbb{R}^d.$$  We list some known properties of this
transform:
\begin{itemize}
\item[i)] The Dunkl transform of a function $f
\in L^1_k( \mathbb{R}^d)$ has the following basic property
\begin{eqnarray*}\| \mathcal{F}_k(f)\|_{\infty,k} \leq
 \|f\|_{ 1,k}\;. \end{eqnarray*}
\item[ii)] The Schwartz space $\mathcal{S}( \mathbb{R}^d)$ is
invariant under the Dunkl transform $\mathcal{F}_k\;.$
\item[iii)] When both $f$ and $\mathcal{F}_k(f)$ are in $L^1_k( \mathbb{R}^d)$,
 we have the inversion formula \begin{eqnarray*} f(x) =   \int_{\mathbb{R}^d}\mathcal{F}_k(f)(y) E_k( ix, y)w_k(y)dy,\quad
x \in \mathbb{R}^d.\end{eqnarray*}
\item[iv)] (Plancherel's theorem) The Dunkl transform on $\mathcal{S}(\mathbb{R}^d)$
 extends uniquely to an isometric isomorphism on
$L^2_k(\mathbb{R}^d)$.
\end{itemize} By i), Plancherel's theorem and the
Marcinkiewicz interpolation theorem (see [19]), we get for $f \in
L^p_k(\mathbb{R}^d)$ with $1\leq p\leq 2$ and $p'$ such that
$\frac{1}{p}+\frac{1}{p'}=1$,
\begin{eqnarray}
\|\mathcal{F}_k(f)\|_{p',k}\leq c\,\|f\|_{p,k}.
\end{eqnarray}
The Dunkl transform of a function in $L^1_k( \mathbb{R}^d)^{rad}$
 is also radial and could be expressed via the Hankel transform. More precisely, according to [14], we have the
following results:
\begin{eqnarray}\mathcal{F}_k(f)(x) &=&\int^{+\infty}_0
\Big( \int_{S^{d-1}}E_k(-ix, y)w_k(y)d\sigma(y)\Big)
F(r)r^{2\gamma+d-1}dr\nonumber\\&=&   \mathcal{H}_{\gamma +
\frac{d}{2}-1} (F)(\|x\|),\quad x \in \mathbb{R}^d, \end{eqnarray}
where $F$ is the function defined on $[ 0, + \infty)$ by $F(\|x\|) =
f(x),\; x \in \mathbb{R}^d$ and $\mathcal{H}_{\gamma +
\frac{d}{2}-1}$ is the Hankel transform of order $\gamma +
\frac{d}{2}-1$.\\ For $\varphi\in\mathcal{S}(\mathbb{R}^d)^{rad}$
and $x\in \mathbb{R}^d\,,$ we have
$\mathcal{F}_k^{-1}(\varphi)(x)=\mathcal{F}_k(\varphi)(-x)=\mathcal{F}_k(\varphi)(x)$.\\

K. Trim\`eche has introduced in [22] the Dunkl translation operators
$\tau_x$, $x\in\mathbb{R}^d$, on $\mathcal{E}( \mathbb{R}^d)\;.$ For
$f\in \mathcal{S}( \mathbb{R}^d)$ and $x, y\in\mathbb{R}^d$, we have
\begin{eqnarray}\mathcal{F}_k(\tau_x(f))(y)=E_k(i x, y)\mathcal{F}_k(f)(y).\end{eqnarray}
 Notice that for all $x,y\in\mathbb{R}^d$,
$\tau_x(f)(y)=\tau_y(f)(x)$ and for fixed $x\in\mathbb{R}^d$
\begin{eqnarray}\tau_x \mbox{\; is a continuous linear mapping from \;}
\mathcal{E}( \mathbb{R}^d) \mbox{\;
into\;}\mathcal{E}(\mathbb{R}^d)\,.\end{eqnarray} As an operator on
$L_k^2(\mathbb{R}^d)$, $\tau_x$ is bounded. A priori it is not at
all clear whether the translation operator can be defined for $L^p$-
functions with $p$ different from 2. However, according to ([18],
Theorem 3.7), the operator $\tau_x$ can be extended to
$L^p_k(\mathbb{R}^d)^{rad},$ $1 \leq p \leq 2$ and we have
\begin{eqnarray}\|\tau_x(f)\|_{p,k} \leq \|f\|_{p,k}\;\;\mbox{for} \;\;f \in
L^p_k(\mathbb{R}^d)^{rad}.\end{eqnarray}

The Dunkl convolution product $\ast_k$ of two functions $f$ and $g$
in $L^2_k(\mathbb{R}^d)$ is given by
$$(f\; \ast_k g)(x) = \int_{\mathbb{R}^d} \tau_x (f)(-y) g(y) w_k(y)dy,\quad
x \in \mathbb{R}^d .$$ The Dunkl convolution product is commutative
and for $f,\,g \in \mathcal{D}( \mathbb{R}^d)$ we have
\begin{eqnarray}\mathcal{F}_k(f\,\ast_k\, g) =
\mathcal{F}_k(f) \mathcal{F}_k(g).\end{eqnarray}It was shown in
([18], Theorem 4.1) that when $g$ is a bounded function in $L^1_k(
\mathbb{R}^d)^{rad}$, then
\begin{eqnarray}(f\; \ast_k g)(x) = \int_{\mathbb{R}^d}  f(y) \tau_x (g)(-y) w_k(y)dy,\quad
x \in \mathbb{R}^d , \end{eqnarray} initially defined on the
intersection of $L^1_k(\mathbb{R}^d)$ and $L^2_k(\mathbb{R}^d)$
extends to $L^p_k(\mathbb{R}^d)$, $1\leq p\leq +\infty$ as a bounded
operator. In particular, \begin{eqnarray}\|f \ast_k g\|_{p,k} \leq
\|f\|_{p,k} \|g\|_{1,k}\,.\end{eqnarray}

In the case $d=1$, $G= \mathbb{Z}_2=\{id,-id\}$ the corresponding
reflection group acting on $\mathbb{R}$ and
$\gamma=k(\alpha)=\alpha+\frac{1}{2}>0$, the Dunkl operator on the
real line is defined by
$$T_1(f)(x)= \frac{df}{dx}(x) + \frac{2\alpha+1}{x} \left[\frac{f(x)
- f(-x)}{2}\right],\quad f \in \mathcal{E}( \mathbb{R}).$$ For
$\lambda \in \mathbb{C}$, the Dunkl kernel is given by
\begin{eqnarray}E_k(\lambda x) = j_\alpha(i\lambda x) + \frac{\lambda x}
{2(\alpha+1)} j_{\alpha+1} (i\lambda x),\quad x \in
\mathbb{R}\end{eqnarray}We have \begin{eqnarray}w_k(x) =
\displaystyle{\frac{|x|^{2\alpha+1}}{2^{\alpha +1}\Gamma(\alpha
+1)}}.\end{eqnarray} For all $x \in \mathbb{R}$, the Dunkl
translation operator $\tau_x$ extends to $L^p_k(\mathbb{R})$,$\, p
\geq 1$ and we have for  $f \in L^p_k(\mathbb{R})$
\begin{eqnarray}\|\tau_x(f)\|_{p,k} \leq 3\|f\|_{p,k}.\end{eqnarray}
\section {Characterization by decomposition }
In this section, we characterize the Besov-Dunkl spaces by
decomposition of functions. Before, we start with some useful
remarks and propositions.
\begin{remark} By ([3], Proposition 1 and 2),
it follows that for $ \beta>0$ and $1\leq p , q \leq + \infty$, the
Besov-Dunkl space $\mathcal{B}\mathcal{D}_{p,q}^{\beta,k}$ is
independent of the choice of the sequence $ (\varphi_{j})_{j\in
\mathbb{N}}$ in $ \Phi(\mathbb{R}^d)$. This space coincide on
$L^p_k(\mathbb{R}^d)$ with the homogeneous Besov-Dunkl space.
\end{remark}
 \begin{remark} It was shown in ([3], Proposition 3) that for $1\leq p , q <+ \infty$ and $\beta>0$, we have
the density of $\mathcal{S}(\mathbb{R}^d)$ in
$\mathcal{B}\mathcal{D}_{p,q}^{\beta,k}.$\end{remark}
 Now, in order to prove
the following propositions, we recall that for $1\leq p\leq 2$,
$\mathfrak{M}_p(\mathbb{R}^d)$ is the class of sequences $
(g_{j})_{j\in \mathbb{N}}$ of functions in $L^p_k(\mathbb{R}^d)$
such that $\mbox{supp}\,\mathcal{F}_k(g_{0}) \subset B(0,2)$ and
$\mbox{supp}\,\mathcal{F}_k(g_{j}) \subset A_j=\{x\in
\mathbb{R}^d\,;\,2^{j-1}\leq \|x\|\leq 2^{j+1}\}$ for $j\in
\mathbb{N\backslash}\{0\}.$
\begin{proposition} Let $\beta>0$,
$1\leq p \leq 2$, $1\leq q < + \infty$ and $ (\varphi_{j})_{j\in
\mathbb{N}}\in \Phi(\mathbb{R}^d)$. If
$f\in\mathcal{B}\mathcal{D}_{p,q}^{\beta,k}$ then one has
\begin{itemize}
\item[i)] $\forall\, n, m\in \mathbb{N}$, $\;f_{n,m}=\displaystyle{\sum_{s=n}^{n+m}} \varphi_s
\ast_k
f\in\mathcal{B}\mathcal{D}_{p,q}^{\beta,k}\cap\mathcal{E}(\mathbb{R}^d).$
\item[ii)] $(\varphi_{j}\ast_k f)_{j\in
\mathbb{N}}\in\mathfrak{M}_p(\mathbb{R}^d)$ \mbox{\quad and \quad}
$f=\displaystyle{\sum_{j=0}^{+\infty}} \varphi_j \ast_k f \,,\,$ in
$L^p_k(\mathbb{R}^d)$.
\end{itemize}\end{proposition}
\begin{proof} Let $ (\varphi_{j})_{j\in
\mathbb{N}}\in \Phi(\mathbb{R}^d)$ and $f\in
\mathcal{B}\mathcal{D}_{p,q}^{\beta,k}$ with $\beta>0$, $1\leq p
\leq 2$ and $1\leq q < + \infty$.\\i) Take for $ n, m\in
\mathbb{N}$, $\;f_{n,m}=\displaystyle{\sum_{s=n}^{n+m}} \varphi_s
\ast_k f$. By the triangle inequality, (2.11) and the property (ii)
for $ ( \varphi_{j})_{j\in \mathbb{N}}$, we have $f_{n,m}\in
L^p_k(\mathbb{R}^d)$. Now, using the properties (i), (ii), (iii) for
$ ( \varphi_{j})_{j\in \mathbb{N}}$, the triangle inequality and
(2.11) again, we get
\begin{eqnarray*}\sum_{j=0}^{+\infty}(2^{j\beta}\| \varphi_j \ast_k f_{n,m}\|_{p,k})^q
&=& \sum_{j=0}^{n+m+1}(2^{j\beta}\| \varphi_j \ast_k
f_{n,m}\|_{p,k})^q\\&\leq& c\sum_{j=0}^{n+m+1}(2^{j\beta}\|
\varphi_j \ast_k f \|_{p,k})^q<+\infty.
\end{eqnarray*} We conclude from (2.7) and (2.10) that
$f_{n,m}\in\mathcal{E}(\mathbb{R}^d)$.
\\ii) Using (2.9) and (2.11) and the property (ii) for $ ( \varphi_{j})_{j\in \mathbb{N}}$, it's clear that $ (\varphi_{j}\ast_k f
)_{j\in \mathbb{N}}$ is in $\mathfrak{M}_p(\mathbb{R}^d)$. For
$r,\,s\in \mathbb{N}$ with $r<s$, we have from the H\"older
inequality that\begin{eqnarray*}
 \|\sum_{j=r}^{s} \varphi_j \ast_k
f\|_{p,k}&\leq &\sum_{j=r}^{s}\| \varphi_j \ast_k f\|_{p,k}\\&\leq
&\Big(\sum_{j=r}^{s} 2^{-j\beta
q'}\Big)^{\frac{1}{q'}}\Big(\sum_{j=r}^{s}(2^{j\beta}\| \varphi_j
\ast_k f\|_{p,k})^q \Big)^{\frac{1}{q}},\end{eqnarray*}where $q'$ is
the conjugate of $q$. Then the series $\displaystyle{
  \sum_{j=0}^{+\infty} \varphi_j \ast_k f} \;$ converges
  in $L^p_k(\mathbb{R}^d)$ and the map $ f \mapsto \displaystyle{
  \sum_{j=0}^{+\infty} \varphi_j \ast_k f} \;$  is
  continuous from $ \mathcal{B}\mathcal{D}_{p,q}^{\beta,k} $ into
  $L^p_k(\mathbb{R}^d)$.\\ Take $\psi \in \mathcal{S}(\mathbb{R}^d)$,
  the series $\displaystyle{
  \sum_{j=0}^{+\infty} \varphi_j \ast_k \psi} \;$
  converges, in particular in $L^2_k(\mathbb{R}^d)$, then by the
  Plancherel theorem and (2.9), we deduce that $\displaystyle{
  \sum_{j=0}^{+\infty} \mathcal{F}_k(\varphi_j)  \mathcal{F}_k(\psi)} \;$ converges in
  $L^2_k(\mathbb{R}^d)$. Since \\$\displaystyle{
  \sum_{j=0}^{+\infty} \mathcal{F}_k(\varphi_j) (x)=1},$ $x\in \mathbb{R}^d$, we
  conclude that $\displaystyle{\sum_{j=0}^{+\infty} \mathcal{F}_k(\varphi_j)  \mathcal{F}_k(\psi)}=
  \mathcal{F}_k(\psi)$, which gives $\displaystyle{\sum_{j=0}^{+\infty} \varphi_j \ast_k \psi=\psi
  } $.  Using
  the fact that
   $\mathcal{S}(\mathbb{R}^d)$ is dense in $
  \mathcal{B}\mathcal{D}_{p,q}^{\beta,k}$ (see Remark 3.2) and $
  \mathcal{B}\mathcal{D}_{p,q}^{\beta,k}$ is continuously included
  in $L^p_k(\mathbb{R}^d)$, we get
  \begin{eqnarray}  f=
  \sum_{j=0}^{+\infty} \varphi_j \ast_k f,\end{eqnarray}which proves the results ii). This completes the
proof. \end{proof} \\\\
For $f\in\mathcal{B}\mathcal{D}_{p,q}^{\beta,k}$, we put $
\|f\|_{\mathcal{B}\mathcal{D}_{p,q}^{\beta,k}}=\displaystyle{
\Big(\sum_{j\in\mathbb{N}}(2^{j\beta}\|\varphi_{j }\ast_k
f\|_{p,k})^q \Big)^{\frac{1}{q}}}.$
\begin{proposition} Let $\beta>0$, $1\leq p \leq 2$ and $1\leq q < +
\infty$. Then
\begin{itemize}
\item[1)]For all compact set $K$ in
$\mathbb{R}^d\backslash\{0\}$, the norms $\|\,.\,\|_{p,k}$ and
$\|\,.\,\|_{\mathcal{B}\mathcal{D}_{p,q}^{\beta,k}}$ are  equivalent
on the set $F_p(K)$ of all functions $f\in L^p_k(\mathbb{R}^d)$ such
that $\mbox{supp}\,\mathcal{F}_k(f)\subseteq K $.\item[2)] The
subspace $F_p$ of functions $f \in L^p_k(\mathbb{R}^d)$ such that
$\mbox{supp}\,\mathcal{F}_k(f)$ is a compact set in
$\mathbb{R}^d\backslash\{0\}$ is dense in
$\mathcal{B}\mathcal{D}_{p,q}^{\beta,k}$.\end{itemize}\end{proposition}\begin{proof}
  1) If $K$ is a compact set in
$\mathbb{R}^d\backslash\{0\}$, then there exists a finite set
$I\subset \mathbb{N}$ such that $K\subset\displaystyle{\bigcup_{j\in
I}} A_j$ where we take $A_0= \{x\in \mathbb{R}^d\,;\, \|x\|\leq 2\}$
when $j=0$. This gives for $f\in F_p(K)$ and $ (\varphi_{j})_{j\in
\mathbb{N}}$ in $ \Phi(\mathbb{R}^d)$ that $\varphi_j \ast_k f=0$,
if $j>1+\mbox{max}\,I$. Hence we obtain
\begin{eqnarray*}\sum_{j=0}^{+\infty}(2^{j\beta}\| \varphi_j \ast_k
f\|_{p,k})^q &\leq &\sum_{j\in I} 2^{j\beta q}\| \varphi_j\|_{1,k}^q
\| f\|_{p,k}^q \\&\leq & c \Big(\sum_{j\in I} 2^{j\beta q}\Big)\|
f\|_{p,k}^q\leq c(K) \| f\|_{p,k}^q .
\end{eqnarray*} Conversely, if $(\varphi_{j})_{j\in \mathbb{N}}$ in
$ \Phi(\mathbb{R}^d)$ and $f\in
\mathcal{B}\mathcal{D}_{p,q}^{\beta,k}$ with
$\mbox{supp}\,\mathcal{F}_k(f)\subseteq K$, then using (3.1) we have
$\displaystyle{f=
  \sum_{j=0}^{+\infty} \varphi_j \ast_k f=
  \sum_{j\in I} \varphi_j \ast_k f}.$ Using H\"older's
inequality, it yields
\begin{eqnarray*}\| f\|_{p,k}\leq \sum_{j\in I}\| \varphi_j \ast_k f\|_{p,k}\leq \Big(\sum_{j\in I} 2^{-j\beta q'}\Big)^{\frac{1}{q'}}\Big(\sum_{j\in\mathbb{N}}(2^{j\beta}\|\varphi_{j }\ast_k
f\|_{p,k})^q \Big)^{\frac{1}{q}}.\end{eqnarray*}
 where $q'$ is the conjugate of $q$.\\2) Assume
$f\in\mathcal{B}\mathcal{D}_{p,q}^{\beta,k}$ and $
(\varphi_{j})_{j\in \mathbb{N}}\in \Phi(\mathbb{R}^d)$, then we put
$f_n=\displaystyle{ \sum_{s=0}^{n} \varphi_s\ast_k f}$ for
$n\in\mathbb{N\backslash}\{0\}.$ It's clear that $f_n \in F_p$. From
(Proposition 3.1, i)), we have
$f_n\in\mathcal{B}\mathcal{D}_{p,q}^{\beta,k}$, hence using the
properties (i) and (iii) for $ ( \varphi_{j})_{j\in \mathbb{N}}$, we
obtain
$$\sum_{j=0}^{+\infty}
2^{j\beta q}\|\varphi_j\ast_k (f-f_n\|_{p,k}^q\leq c \sum_{|j|\geq
n} 2^{j\beta q}\|\varphi_j\ast_k f\|_{p,k}^q.$$ Since
$f\in\mathcal{B}\mathcal{D}_{p,q}^{\beta,k}$, then we deduce that
$\;\lim_{n\rightarrow+\infty} f_n =f \mbox{\quad
in\quad}{\mathcal{B}\mathcal{D}_{p,q}^{\beta,k}}.$
\end{proof}
\\\\Let $1\leq p\leq 2$ and $1\leq q<+\infty$. In order to prove the
following theorem, we denote by
$$\|g_j \|^\star=\Big(\displaystyle{\sum_{j\in\mathbb{N}}}(2^{j\beta}\|g_{j }
\|_{p,k})^q\Big)^{\frac{1}{q}},$$ for any sequence $(g_{j})_{j\in
\mathbb{N}}$ of functions in $ L^p_k(\mathbb{R}^d)$.\begin{theorem}
If $\beta>0$, $1\leq p \leq 2$ and $1\leq q < + \infty$, then
\begin{eqnarray}\mathcal{B}\mathcal{D}_{p,q}^{\beta,k}=\Big\{f \in L^p_k(\mathbb{R}^d)&:& \exists (g_{j})_{j\in
\mathbb{N}}\in\mathfrak{M}_p(\mathbb{R}^d)\;\mbox{such
that}\;f=\sum_{j=0}^{+\infty} g_j\nonumber \\&&\mbox{in}\;
L^p_k(\mathbb{R}^d) \;\mbox{and}\; \|g_j \|^\star <+\infty
\Big\}.\end{eqnarray}\end{theorem} \begin{proof} (Proposition 3.1,
ii)) shows that every $f\in \mathcal{B}\mathcal{D}_{p,q}^{\beta,k}$
can be represented by the right-hand side of (3.2). Conversely, if
$f \in L^p_k(\mathbb{R}^d)$ is given by $$f=\sum_{j=0}^{+\infty} g_j
\quad\mbox{in}\;  L^p_k(\mathbb{R}^d), \;(g_{j})_{j\in
\mathbb{N}}\in\mathfrak{M}_p(\mathbb{R}^d) \quad\mbox{and}\quad
\|g_j \|^\star <+\infty,$$ then we have for $ (\varphi_{j})_{j\in
\mathbb{N}}\in \Phi(\mathbb{R}^d)$, $$ \varphi_j\ast_k f =
\varphi_j\ast_k(g_{j-1}+g_{j}+g_{j+1}),\;j\in \mathbb{N}$$ where we
put for convenience $g_{-1}=0$. Then by (2.11) and H\"older's
inequality for $j\in \mathbb{N}$, we obtain
\begin{eqnarray*} \| \varphi_j\ast_k f \|_{p,k}^q\leq c\, 3^{q-1}
\sum_{s=j-1}^{j+1}\| g_s \|_{p,k}^q.
\end{eqnarray*} Thus summing over $j$ with weights $2^{j \beta q}$ and using the triangle inequality, we get
  \begin{eqnarray*}
 \sum_{j\in\mathbb{N}}(2^{j\beta}\|\varphi_{j }\ast_k f\|_{p,k})^q\leq
 c \sum_{j\in\mathbb{N}}(2^{j\beta}\|g_{j } \|_{p,k})^q<+\infty.
   \end{eqnarray*} Thus we obtain the result. \end{proof}
\begin{remark} Put $ \mathcal{A}$ the set of all functions $\phi \in
\mathcal{S}(\mathbb{R}^d)^{rad}$ such that
$$\,\mbox{supp}\,\mathcal{F}_k(\phi) \subset \Big\{x\in
\mathbb{R}^d\,;\,1\leq \|x\|\leq 2\} \Big\} $$ and denote by $
\mathcal{C}_{ p,q}^{\phi ,\beta,k} $, $ \beta>0$, $1\leq p , q \leq
+ \infty$, the subspace of functions $f \in L^p_k(\mathbb{R}^d)$
satisfying
\begin{eqnarray*}
\Big(\int^{+ \infty}_0 \left(\frac{\|f \ast_k \phi_t
\|_{p,k}}{t^\beta}\right)^q \, \frac{dt}{t}\Big)^{\frac{1}{q}} &<& +
\infty \;\;\;\quad if\quad q < +\infty
\end{eqnarray*} (usual modification when $q=+\infty$),
where
$\phi_t(x)=\frac{1}{t^{2(\gamma+\frac{d}{2})}}\phi(\frac{x}{t})$,
 for all $t\in (0,+\infty)$ and $x\in\mathbb{R}^d$. Then from ([3], Theorem 2),
 we have for $ \beta>0$, $1\leq p , q \leq + \infty$ and all $\phi \in\mathcal{A}$
  \begin{eqnarray}
\mathcal{B}\mathcal{D}_{p,q}^{\beta,k}\subset \mathcal{C}_{
p,q}^{\phi ,\beta,k}.\end{eqnarray}   In the case $d=1$, $G=
\mathbb{Z}_2$, $\gamma=k(\alpha)=\alpha+\frac{1}{2}>0$ and
$$T_1(f)(x)= \frac{df}{dx}(x) + \frac{2\alpha+1}{x} \left[\frac{f(x)
- f(-x)}{2}\right],\quad f \in \mathcal{E}( \mathbb{R}),$$ we can
characterize the Besov-Dunkl spaces by differences using the modulus
of continuity of second order of $f$. Put $\mathcal{H}$ the set of
all functions $ \phi \in \mathcal{S}_\ast(\mathbb{R})$ such that
$\displaystyle{\int_0^{+\infty}\phi(x)d\mu_\alpha(x)=0 }$ with
$d\mu_\alpha(x) = \displaystyle{\frac{|x|^{2\alpha+1}}{2^{\alpha
+1}\Gamma(\alpha +1)}\; dx}$ and $\mathcal{S}_\ast(\mathbb{R})$ the
space of even Schwartz functions on $\mathbb{R}$. Then we can assert
from ([1], Theorem 3.6) that for $0<\beta<1$, $1< p<+\infty$, $1\leq
q\leq+\infty$ and all $\phi \in\mathcal{H}$, that \begin{eqnarray}
\mathcal{C}_{ p,q}^{\phi
,\beta,k}=BD^{p,q}_{\alpha,\beta},\end{eqnarray} where
$BD^{p,q}_{\alpha,\beta}$ is the subspace of functions $f \in L^p(
\mu_\alpha)$ satisfying
\begin{eqnarray*}\Big(\int^{+ \infty}_0
\left(\frac{\omega_{p}(f)(t)}{x^\beta}\right)^q \,
\frac{dx}{x}\Big)^{\frac{1}{q}} < + \infty \;\;\;\quad
\mbox{if}\quad q < +\infty
\end{eqnarray*} and \begin{eqnarray*}  \sup_{x\in (0,+\infty)}\frac{\omega_{p}(f)(t)}{x^\beta} &<& +\infty\,\quad\quad if \quad q=+\infty.\end{eqnarray*}
 Note that when $d=1$, we have $ \mathcal{A}\subset\mathcal{H}$, then from (3.3) and (3.4),
 we obtain for $0<\beta<1$, $1< p<+\infty$, $1\leq
q\leq+\infty$ and all $\phi \in\mathcal{A}$
$$\mathcal{B}\mathcal{D}_{p,q}^{\beta,k}\subset
BD^{p,q}_{\alpha,\beta}.$$\\From (3.4), it's clear that the space $
\mathcal{C}_{ p,q}^{\phi ,\beta,k}$ is independent of the specific
selection of $\phi$ in $\mathcal{H}$. In particular, if we take for
example the function $\phi$ defined on $\mathbb{R}$ by
$$\phi(x)=-x\varphi'(x)-2(\alpha+1)\varphi(x),$$ where $\varphi$ is the
Gaussian function, $\varphi(x)=e^{-\frac{x^2}{2}}$, then we can see
that $\phi\in\mathcal{H}$. The dilation $\phi_t$ of $\phi$ gives
$\phi_t(x)=t\frac{d}{dt}\varphi_t(x)$, for $x\in\mathbb{R}$ and
$t\in(0,+\infty)$. From (3.4), it yields that for $0<\beta<1$, $1<
p<+\infty$ and $1\leq q\leq+\infty$, $$f\in
BD^{p,q}_{\alpha,\beta}\Longleftrightarrow  \int^{+ \infty}_0
\left(\frac{\|t\frac{d}{dt}(f \ast_k \varphi_t)
\|_{p,k}}{t^\beta}\right)^q \, \frac{dt}{t}  <+ \infty,$$(usual
modification when $q=+\infty$).
\end{remark}
\section{Moduli of continuity and Dunkl tranform}
In this section, we obtain inequalities for the Dunkl transform
$\mathcal{F}_k(f) $ of $L^p$-function $f$, $1\leq p\leq2$, both on
the real line and in radial case on $\mathbb{R}^d$. As consequence,
we give a quantitative form of the Riemann-Lebesgue
lemma and further results of integrability for the Dunkl transform.\\

Throughout this section, we denote by $p'$ the conjugate of $p$.
According to (2.8) and (2.14), we
recall that for $1\leq p\leq2$, $\omega_{p}(f)$ is the modulus of continuity of $f$ and is given :\\
$\bullet$ On the real line by :$$\omega_{p}(f)(t)= \|\tau_t(f) +
\tau_{-t}(f) - 2f\|_{p,k}\,,\;\;t\geq0\,,\;\;f\in
L^p_k(\mathbb{R}).$$ $\bullet$ In radial case on $\mathbb{R}^d$ by :
$$\omega_p(f) (t) = \displaystyle{\int_{S^{d-1}}}\|\tau_{tu}(f) -
f\|_{p,k}d\sigma(u)\,,\;\;t\geq0\,,\;\;f\in L^p_k(
\mathbb{R}^d)^{rad}.$$
\begin{lemma} (see [6, 7]) Let $\alpha>-\frac{1}{2}$. Then there exist positive
constants $c_{1,\alpha}$ and $c_{2,\alpha}$ such that
\begin{eqnarray} c_{1,\alpha} \min\{1,(\lambda t)^2\}\leq 1-j_\alpha(\lambda t)\leq c_{2,\alpha} \min\{1,(\lambda
t)^2\},\quad t, \lambda \in \mathbb{R}.
\end{eqnarray}
\end{lemma}
\begin{theorem} Let $1\leq p \leq 2$ and $f \in L^p_k(\mathbb{R})$. Then
there exists a positive constant $c$ such that for any
$t\in(0,+\infty)$, one has $$\Big(\int_{
\mathbb{R}}\min\{1,(t|x|)^{2p'}\}\,|\mathcal{F}_k(f)(x)|^{p'}
w_k(x)dx\Big)^{\frac{1}{p'}}\leq c\,\omega_p(f) (t)\,,
\;\;\mbox{if}\;\; 1<p\leq2\,,$$ $$ess\sup_{x \in
\mathbb{R}}\Big[\min\{1,(tx)^{2}\}\,|\mathcal{F}_k(f)(x)|\Big]\leq
c\,\omega_1(f) (t)\,, \;\;\mbox{if}\;\; p=1.$$\end{theorem}
\begin{proof} For $f \in
L^p_k(\mathbb{R})$, we have by (2.6)
$$\mathcal{F}_k(\tau_t(f)+\tau_{-t}(f) -2f)(x )=
[E_k(itx)+E_k(-itx) -2]\mathcal{F}_k(f)(x) \,,$$ for $t \in (0,
+\infty)$ and a.e $x\in \mathbb{R}$. Applying (2.4), we get\\\\
$\|\mathcal{F}_k(\tau_t(f)+\tau_{-t}(f) -2f)\|_{p',k}$
\begin{eqnarray} &=&
\Big(\int_{ \mathbb{R}}|\mathcal{F}_k(f)(x)|^{p'}|E_k(itx)+E_k(-itx)
-2|^{p'} w_k(x)dx\Big)^{\frac{1}{p'}}\nonumber\\&\leq&
c\;\omega_{p}(f)(t)\,.\end{eqnarray} From (2.12), it yields
\begin{eqnarray}
|E_k(itx)+E_k(-itx) -2| &\geq& 2 |j_\alpha(tx) -1|
\end{eqnarray} then using (4.1) and (4.3) in (4.2), we obtain the result.
Here when $p=1$, we make the usual modification. \end{proof}
\begin{remark} Note that if $p=2$, then by Plancherel's theorem and
(4.1), there exist positive constants $c_1,c_2$ such that
$$c_1\,\omega_2(f) (t)\leq\Big(\int_{
\mathbb{R}}\min\{1,(tx)^{4}\}\,|\mathcal{F}_k(f)(x)|^2
w_k(x)dx\Big)^{1/2}\leq c_2\,\omega_2(f)(t). $$
\end{remark}

As consequence immediate of the theorem 4.1, we obtain the following
quantitative form of the Riemann-Lebesgue lemma.
\begin{corollary} Let $1\leq p \leq 2$ and $f \in L^p_k(\mathbb{R})$. Then
there exists a positive constant $c$ such that for any
$t\in(0,+\infty)$, one has
$$\Big(\int_{|x|>\frac{1}{t}}\,|\mathcal{F}_k(f)(x)|^{p'}
w_k(x)d(x)\Big)^{\frac{1}{p'}}\leq c\,\omega_p(f) (t)\,,
\;\;\mbox{if}\;\; 1<p\leq2\,,$$
$$ess\sup_{|x|>\frac{1}{t}} |\mathcal{F}_k(f)(x)|\leq
c\,\omega_1(f) (t)\,, \;\;\mbox{if}\;\; p=1.$$
\end{corollary}
\begin{theorem} Let $\alpha>-\frac{1}{2}$, $\beta>2(\alpha+1)$, $A>0$ and $f\in
L^1_k(\mathbb{R})$. If $f$ satisfies \begin{eqnarray}\sup_{t \in
(0,+\infty)}\; \frac{\omega_1(f) (t)}{t^\beta} < A\;,\end{eqnarray}
then $$\mathcal{F}_k(f) \in L^1_k(\mathbb{R}).$$
\end{theorem}
\begin{proof} From the theorem 4.1 and (4.4), we obtain
\begin{eqnarray} ess\sup_{|x|\leq\frac{1}{t}} (tx)^{2}|\mathcal{F}_k(f)(x)|\leq
c\,\omega_1(f) (t)\leq c\,t^\beta \end{eqnarray} By H\"older's
inequality, (4.5) and (2.13), we have
\begin{eqnarray*}\int_{{|x|\leq
\frac{1}{t}}}|x|\,|\mathcal{F}_k(f)(x)| \, w_k(x)dx &\leq&
ess\sup_{|x|\leq\frac{1}{t}}
x^{2}|\mathcal{F}_k(f)(x)|\int_{|x|\leq\frac{1}{t}}|x|^{-1}
w_k(x)dx\\&\leq& c\;t^{\beta-2}\, \int_{0}^{\frac{1}{t}}
x^{2\alpha}\, dx \leq\; c\; t^{ \beta - 2(\alpha+1)-1}.
\end{eqnarray*} Integrating with respect to $t$ over $(0,1)$
and applying Fubini's theorem, we obtain $$\int_{|x|\geq 1}
|\mathcal{F}_k(f)(x)| w_k(x)dx \leq c\int^{1}_0 t^{ \beta -
2(\alpha+1)-1}dt <+\infty\;.$$ Since $L^{\infty}_k([-1,1],w_k(x)dx)
\subset L^{1}_k([-1,1],w_k(x)dx)$, we deduce that $\mathcal{F}_k(f)$
is in $L^{1}_k(\mathbb{R})$.
\end{proof}
\begin{theorem} Let $1\leq p \leq 2$ and $f \in L^p_k(
\mathbb{R}^d)^{rad}$. Then there exists a positive constant $c$ such
that for any $t\in(0,+\infty)$, one has $$\Big(\int_{
\mathbb{R}^d}\min\{1,(t\|x\|)^{2p'}\}\,|\mathcal{F}_k(f)(x)|^{p'}
w_k(x)dx\Big)^{\frac{1}{p'}}\leq c\,\omega_p(f) (t)\,,
\;\;\mbox{if}\;\; 1<p\leq2\,,$$ $$ess\sup_{x \in
\mathbb{R}^d}\Big[\min\{1,(t\|x\|)^{2}\}\,|\mathcal{F}_k(f)(x)|\Big]\leq
c\,\omega_1(f) (t)\,, \;\;\mbox{if}\;\; p=1.$$
\end{theorem}
\begin{proof}  Let $f \in L^p_k(
\mathbb{R}^d)^{rad},$ we can write from (2.6)
$$\mathcal{F}_k(\tau_{tu}(f)-f)(x) = \mathcal{F}_k(f)(x) [E_k(itu,x)-1]\;,
 $$ for $u \in S^{d-1}$, $t \in (0,+\infty)$ and a.e $x
\in \mathbb{R}^d$. Applying (2.4), we get
\begin{eqnarray}\|\mathcal{F}_k(\tau_{tu}(f)-f))\|_{p',k} &=&
\Big(\int_{\mathbb{R}^d}|\mathcal{F}_k(f)(x)|^{p'}
|E_k(itu,x)-1|^{p'} w_k(x)dx\Big)^{\frac{1}{p'}}\nonumber\\&\leq&
c\;\|\tau_{tu}(f) - f\|_{p,k}\,.
\end{eqnarray}
According to (2.1) and (2.5), we have \\ \\
$\displaystyle{\int_{\mathbb{R}^d} |\mathcal{F}_k(f)(x)|^{p'}
|E_k(itu,x)-1|^{p'}  w_k(x)}dx$
\begin{eqnarray*}
= d_k \int^{+ \infty}_0 |\mathcal{H}_{\gamma +
\frac{d}{2}-1}(F)(r)|^{p'}\Big(\int_{S^{d-1}}|E_k(itru,z)-1|^{p'}
w_k(z)d\sigma(z)\Big)r^{2\gamma+d-1}dr, \end{eqnarray*}
 where $F$ is a function
on $(0,+ \infty)$ such
that $F(\|x\|) = f(x)$, for all $x \in \mathbb{R}^d$.\\
On the other hand, by (2.2), (2.3) and H\"older's inequality, we
get\\$d_k |j_{\gamma + \frac{d}{2}-1}(rt)-1|$
 \begin{eqnarray*}&=& |\int_{S^{d-1}}[E_k(i
tru,z)-1]w_k(z)d\sigma(z)|\nonumber\\
 &\leq&\Big(\int_{S^{d-1}}w_k(z)d\sigma(z)\Big)^{\frac{1}{p}}
\Big(\int_{S^{d-1}}|E_k(i tru,z)-1|^{p'} w_k(z)d\sigma(z)\Big)^{\frac{1}{p'}}\nonumber\\
 &\leq& d_k^{\frac{1}{p}} \Big(\int_{S^{d-1}}|E_k(i tru,z)-1|^{p'} w_k(z)d\sigma(z)\Big)^{\frac{1}{p'}},
\end{eqnarray*}
hence we obtain,
\begin{eqnarray}|j_{\gamma + \frac{d}{2}-1} (rt) - 1|^{p'} \leq c
 \int_{S^{d-1}}|E_k(itru,z)-1|^{p'} w_k(z)d\sigma(z).\end{eqnarray}
 From (4.6) and (4.7), it follows that
\begin{eqnarray*}\int^{+ \infty}_0 |\mathcal{H}_{\gamma + \frac{d}{2} -1}
 (F)(r)|^{p'}
 |j_{\gamma + \frac{d}{2} -1}(rt)-1|^{p'}
 r^{2\gamma+d-1}dr \leq c\,\|\tau_{tu}(f) - f\|_{p,k}^{p'} \,,\end{eqnarray*}
 So using (4.1), (2.1) and (2.5), we obtain \begin{eqnarray*}\int_{
\mathbb{R}^d}\min\{1,(t\|x\|)^{2p'}\}\,|\mathcal{F}_k(f)(x)|^{p'}
w_k(x)dx \leq c\,\|\tau_{tu}(f) - f\|_{p,k}^{p'} \;\end{eqnarray*}
and we deduce our result. When $p=1$, we make the usual
modification.\end{proof}
\begin{remark} Note that if $p=2$, by Plancherel's theorem and
(4.1), there exist positive constants $c_1,c_2$ such that
$$c_1\,\omega_2(f) (t)\leq\Big(\int_{
\mathbb{R}^d}\min\{1,(t\|x\|)^{4}\}\,|\mathcal{F}_k(f)(x)|^2
w_k(x)d(x)\Big)^{1/2}\leq c_2\,\omega_2(f)(t).$$
\end{remark}

As consequence of the theorem 4.3, we obtain the following
quantitative form of the Riemann-Lebesgue lemma.
\begin{corollary} Let $1\leq p \leq 2$ and $f \in L^p_k(
\mathbb{R}^d)^{rad}$. Then there exists a positive constant $c$ such
that for any $t\in(0,+\infty)$, one has
$$\Big(\int_{\|x\|>\frac{1}{t}}\,|\mathcal{F}_k(f)(x)|^{p'}
w_k(x)d(x)\Big)^{\frac{1}{p'}}\leq c\,\omega_p(f) (t)\,,
\;\;\mbox{if}\;\; 1<p\leq2\,,$$
$$ess\sup_{\|x\|>\frac{1}{t}} |\mathcal{F}_k(f)(x)|\leq
c\,\omega_1(f) (t)\,, \;\;\mbox{if}\;\; p=1.$$
\end{corollary}
\begin{theorem} Let $\beta>2(\gamma+\frac{d}{2})$, $A>0$ and $f\in
L^1_k(\mathbb{R}^d)^{rad}$. If $f$ satisfies \begin{eqnarray}\sup_{t
\in (0,+\infty)}\; \frac{\omega_1(f) (t)}{t^\beta} <
A\;,\end{eqnarray} then $$\mathcal{F}_k(f) \in
L^1_k(\mathbb{R}^d)^{rad}.$$
\end{theorem}
\begin{proof} From the theorem 4.3 and (4.8), we obtain
\begin{eqnarray} ess\sup_{\|x\|\leq\frac{1}{t}} (t\|x\|)^{2}|\mathcal{F}_k(f)(x)|\leq
c\,\omega_1(f) (t)\leq c\,t^\beta \end{eqnarray} By H\"older's
inequality, (4.9) and (2.1), we
have\\\\$\displaystyle{\int_{{\|x\|\leq
\frac{1}{t}}}\|x\|\,|\mathcal{F}_k(f)(x)| \, w_k(x)dx}$
\begin{eqnarray*} &\leq&
ess\sup_{\|x\|\leq\frac{1}{t}}
\|x\|^{2}|\mathcal{F}_k(f)(x)|\int_{\|x\|\leq\frac{1}{t}}\|x\|^{-1}
w_k(x)dx\\&\leq& c\;t^{\beta-2}\, \int_{0}^{\frac{1}{t}}
r^{2\gamma+d-2}\, dr \leq\; c\; t^{ \beta -
2(\gamma+\frac{d}{2})-1}.
\end{eqnarray*}  Integrating with respect to $t$ over $(0,1)$
and applying Fubini's theorem, we obtain $$\int_{\|x\|\geq 1}
|\mathcal{F}_k(f)(x)| w_k(x)dx \leq c\int^{1}_0 t^{ \beta -
2(\gamma+\frac{d}{2})-1}dt <+\infty\;.$$ Since
$L^{\infty}_k(B(0,1),w_k(x)dx) \subset L^{1}_k(B(0,1),w_k(x)dx)$, we
deduce that $\mathcal{F}_k(f)$ is in $L^{1}_k(\mathbb{R}^d)$.
\end{proof}

\end{document}